\documentclass[11pt, oneside]{amsart}

\usepackage{geometry}
\usepackage{latexsym}
\usepackage{amsmath}
\usepackage{amsfonts}
\usepackage{amssymb}
\usepackage[dvips]{graphicx}
\usepackage{color}
\usepackage{float}
\usepackage{tikz}

\newtheorem{theorem}{Theorem}[section]
\newtheorem*{lemma*}{Lemma}
\newtheorem*{theorem*}{Theorem}
\newtheorem{lemma}[theorem]{Lemma}

\newtheorem{proposition}[theorem]{Proposition}

\theoremstyle{definition}

\newtheorem{example}[theorem]{Example}

\theoremstyle{remark}

\numberwithin{equation}{section}

\begin{document}

 \title{Left-orderability and cyclic branched coverings}
\author{Ying Hu}
\address{Department of Mathematics, Louisiana State University, Baton Rouge, Louisiana 70803}
\email{yhu4@math.lsu.edu}
 \date{\today}
 
 \maketitle

\begin{abstract}
We provide an alternative proof of a sufficient condition for the fundamental group of the $n^{th}$ cyclic branched cover of $S^3$ along a prime knot $K$ to be left-orderable, which is originally due to  Boyer-Gordon-Watson.  As an application of this sufficient condition, we show that for any $(p,q)$ two-bridge knot, with $p\equiv 3 \text{ mod } 4$, there are only finitely many cyclic branched covers whose fundamental groups are not left-orderable. This answers a question posed by D{\c a}bkowski, Przytycki and Togha.

\end{abstract}

\section{Introduction}

\subsection{Background and results}
 A group $G$ is called {\it left-orderable} if there exists a strict total ordering $<$ on the set of group elements, such that  given any two elements $a$ and $b$ in $G$, if $a<b$ then $c a < c b$ for any $c\in G$.

It is known that  any connected, compact, orientable
 $3$-manifold with a positive first  Betti number  has a left-orderable
 fundamental group \cite[Theorem 1.1]{brw}\cite{hs}. In contrast, for a rational homology sphere, the left-orderability of its fundamental group is a nontrivial property, which is  closely related to the co-oriented taut foliations on the manifold \cite{cd}.  Moreover,   Boyer, Gordon and Watson conjectured that an irreducible rational homology $3$-sphere $M$ is an $L$-space \cite{os} if and only if its fundamental group $\pi_1(M)$ is not left-orderable \cite{bgw}.

Let $X_K$ be the complementary space obtained by removing an open tubular neighborhood of the knot $K$ from the three sphere  $S^3$ and $X_K^{(n)}$ be the $n$th cyclic branched cover of $S^{3}$ branched over the knot $K$. The first Betti number $b_1(X_K^{(n)})$ equals zero  if no root of the Alexander polynomial $\Delta_K(t)$ is an $n^{th}$ root of unity. Hence, most of the cyclic branched covers along a knot are rational homology spheres.  In particular, this is the case if $n$ is a prime power.

For this class of rational homology spheres, the L-space conjecture \cite{bgw} has been verified in the following cases, where they are all L-spaces and have non-left-orderable fundamental groups:
\begin{itemize}
 \item The twofold branched cover of any non-split alternating link \cite{bgw,gre,ito,os1}. 
 \item The $n^{th}$ cyclic branched cover of a $(p,q)$ two-bridge knot with $p/q=2m+\frac{1}{2k}$, $mk>0$ and $n$ arbitrary \cite{dpt, tp}. 
 \item  The $3^{rd}$ and $4^{th}$ cyclic branched cover of a $(p,q)$ two-bridge knot with $p/q=n_1+\frac{1}{1+\frac{1}{n_2}}$ and $n_1, n_2$ are positive odd integers (i.e.  $p/q=2m+\frac{1}{2k}$, $mk<0$) \cite{dpt,gl,tp,ter}. 
\end{itemize}

The motivation of this paper is  a question posed in \cite{dpt}: Given a two-bridge knot $K$,  is  $\pi_1(X_K^{(n)})$ always non-left-orderable whenever $b_1(X_K^{(n)})=0$ ?  We answer this question negatively. In fact, we prove that  for $(p,q)$ two-bridge knots with $p\equiv 3$ mod $4$, there are only finitely many cyclic branched covers that have non-left-orderable fundamental groups. At the end, we will present the knot  $5_2$ as an example and show that the fundamental group $\pi_1(X_{5_2}^{(n)})$ is left-orderable if $n\geq9$. Shortly after this article posted, Tran computed a lower bound (depending on the knot) on the order $n$ so that  the $n^{th}$ cyclic branched cover has a non-left-orderable fundamental group for a large class of two-bridge knots \cite{tran}.

A similar question for hyperbolic knots was also posed in \cite{dpt} and was first answered in \cite[Proposition 23]{clw}.  They showed that the twofold branched cover  of $S^3$ along the Conway knot, which is a non-alternating hyperbolic knot listed as 11n34 in the standard knot tables, has a left-orderable fundamental group and so do all even order cyclic branched covers.

\subsection{Plan of the paper.} 
{\bf Section 2} is devoted to proving Lemma \ref{lem:in}, which is essential in our proof of Theorem \ref{thm:cri}. 

 \begin{lemma*}[Lemma \ref{lem:in}]
   Given a  knot $K$ in $S^3$,  denote by $Z$ a meridional element in the knot group $\pi_1(X_K)$. Suppose that there exists a group homomorphism $\rho$ from $\pi_1(X_K)$ to a group $G$ and  $\rho(Z^{n})$ is in the center of $G$.  Then  $\rho$ induces a group homomorphism from $\pi_1(X_K^{(n)})$ to $G$. In particular, if $\rho$ is non-abelian, then the induced homomorphism is nontrivial. 
 \end{lemma*}

We finish the proof of Theorem \ref{thm:cri} in {\bf Section 3}.
 \begin{theorem*}[Theorem \ref{thm:cri}]
  Given any prime knot $K$ in $S^3$, denote by  $Z$ a meridional element of 
 $\pi_1(X_K)$. If there exists a non-abelian representation $\pi_1(X_K)$ to $SL(2,\mathbb{R})$ such that  $Z^n$ is sent to $\pm I$ then the fundamental group $\pi_1(X_K^{(n)})$ is left-orderable. 
\end{theorem*}

This result was first observed by Boyer-Gordon-Watson in \cite{bgw}, where they showed the following: 
\begin{theorem*}[Theorem 6 in \cite{bgw}]
   Let $K$ be a prime knot in the $3$-sphere and suppose that the fundamental group of its twofold branched cyclic cover is not left-orderable. If $\rho: \pi_1(S^3\setminus K)\rightarrow Homeo_+(S^1)$ is a homomorphism such that $\rho(\mu^2)=1$ for some meridional class $\mu$ in $\pi_1(S^3\setminus K)$, then the  image of $\rho$ is either trivial or isomorphic to $\mathbb{Z}_2$.
\end{theorem*}
 
 Here we make two remarks in comparison to Theorem \ref{thm:cri} with \cite[Theorem 6]{bgw}.
 \begin{itemize}
  \item The proof of \cite[Theorem 6]{bgw} naturally extends to the $n^{th}$ cyclic branched cover for arbitrary $n$. Since $PSL(2,\mathbb{R})$ is a subgroup of $Homeo_+(S^1)$, the group of orientation preserving homeomorphisms of $S^1$, Theorem \ref{thm:cri} is contained in \cite[Theorem 6]{bgw} in this sense.
 \item On the other hand,  if we replace the central extension  
 $$0 \longrightarrow \mathbb{Z} \longrightarrow \widetilde{SL(2,\mathbb{R})}\longrightarrow SL(2,\mathbb{R})\longrightarrow 1$$
that we used in the proof of Theorem \ref{thm:cri} by the extension below \cite{ghy}
 \begin{displaymath}
   0 \longrightarrow \mathbb{Z} \longrightarrow \widetilde{Homeo_+(S^1)}\longrightarrow Homeo_+(S^1)\longrightarrow 1
 \end{displaymath}
 the same statement with \cite[Theorem 6]{bgw} can be achieved. 
\end{itemize}

  Finally,  in {\bf Section 4}, we prove our main result in this paper. 
    \begin{theorem*}[Theorem \ref{thm:tb}]
Given a $(p,q)$ two-bridge knot $K$, with $p\equiv 3 \text{ mod } 4$, there are only  finitely many cyclic branched covers, whose fundamental groups  are not left-orderable.
\end{theorem*}

\section{The fundamental groups of  cyclic branched covers}

Given a Seifert surface $F$, one can present the knot group $\pi_1(X_K)$ as an HNN extension of $\pi_1(S^3\setminus F)$ over the surface group $\pi_1(F)$, (the usual definition of the HNN extension requires $F$ to be incompressible, but we do not need it here). We then apply the Reidemeister-Schreier Method to the presentation of $\pi_1(X_K)$ and obtain a presentation of $\pi_1(X_K^{(n)})$, from where Lemma \ref{lem:in} follows.

More precisely, let $F$ be a Seifert surface of an oriented knot $K$. It has a regular neighborhood that is homeomorphic to $F\times [-1, 1]$, where the positive direction is chosen so that the induced orientation on the boundary $\partial F$ is the same as the chosen orientation on the knot $K$.

\begin{center}
\begin{figure}[ht]
\begin{tikzpicture}[scale=0.8]
 \draw [ultra thick] (2,2) -- (9,2);
 \draw [ultra thick] (2,4) -- (9,4);
 \draw [ultra thick]  (10.5,3) circle [radius=1.8];
 \draw [->, ultra thick, cyan] (9,2) to [out=298, in=180] (10.5,1) to  [out=0, in=270] (12.6,3);
 \draw [ultra thick, cyan](12.6,3) to [out=90, in=0] (10.5,5) to [out=180, in=62] (9,4);
 \draw [dashed, thick] (2,3) -- (10.5,3);
 \draw [ultra thick] (12,2.18) -- (12,2) -- (12.18,2);
 \node[right] at (10.5,3) {$K$};
 \node[left] at (2,2) {$F_-$};
  \node[left] at (2,3) {$F$};
   \node[left] at (2,4) {$F_+$};
   \node[above left] at (9,4) {$P_+$};
   \node[below left] at (9,2) {$P_-$};
   \node[cyan] at (13, 3) {$C$};
   \node [below left] at (12,2.3) {$Z$};
 \draw [fill=black] (10.5,3) circle [radius=0.1];
 \draw [fill=black] (9,2) circle [radius=0.1];
 \draw [fill=black] (9,4) circle [radius=0.1]; 
\end{tikzpicture}
\caption{\small A cross-sectional view of a collar neighborhood of $F$ in the knot complement $X_K$, where   $F_{\pm}$ represent $F\times \pm 1$, respectively.  In addition, the point $P_+$ (resp. $P_-$) is the intersection point of  the meridian $Z$ and $F_+$ (resp. $F_-$).
}
\label{fig:seifert}
\end{figure}
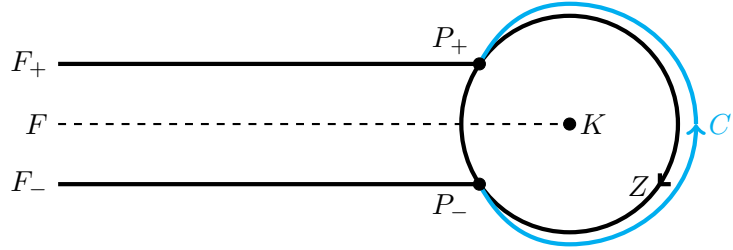
\end{center}

Suppose that the free group $\pi_1(F_-,P_- )$ is generated by  $\{a^-_i\}_{ i=1,\dots, 2g}$ and $\pi_1(F_+, P_+)$ is generated by $\{a^+_i\}_{ i=1, \dots, 2g}$, where $g$ is the genus of the Seifert surface $F$. 
 
We denote by $\alpha_i^-$ the image of $a_i^-$ under the inclusion map $$\pi_1(F_-, P_-) \rightarrow \pi_1(S^3-F,P_-)$$ 
and denote by  $\alpha_i^+$  the image of $a_i^+$ in $\pi_1(S^3-F, P_-)$ under the composition map 
\begin{displaymath}
  \pi_1(F_+,P_+) \rightarrow \pi_1(S^3-F, P_+)\rightarrow \pi_1(S^3-F, P_-), 
\end{displaymath}
where the second map from $\pi_1(S^3-F, P_+)$ to $\pi_1(S^3-F, P_-)$ is the isomorphism induced by the arc $C$ connecting $P_-$ to $P_+$ as depicted in Figure $1$. 
By the Van Kampen Theorem,  we have 
 \begin{align}
  &\pi_1(X_K,P_-) = \label{equ:vkt} \\
  &\quad\pi_1(S^3-F, P_-)\ast <Z> / \ll Z \alpha_i^{+} Z^{-1}=\alpha_i^{-}, i=1, \dots, 2g \gg. \nonumber
 \end{align}

If the complement of the Seifert surface $F$ in $S^3$ is also a handlebody, which is always the case when $F$ is constructed through Seifert's algorithm,  then the group $\pi_1(S^3-F,P_-)$ is also free and  we assume that 
\begin{displaymath}
  \pi_1(S^3-F,P_-)=<x_1,\dots, x_{2g}>.
\end{displaymath}
In this case, from (\ref{equ:vkt}), we obtain  Lin's presentation for the knot group $\pi_1(X_K, P_-)$ \cite[Lemma 2.1]{lin} as follows: 
\begin{equation}
  \pi_1(X,P_-)=<x_1, x_2, \dots, x_{2g-1}, x_{2g}, Z: Z \alpha_i^+ Z^{-1}=\alpha_i^-,  i=1, \dots, 2g>, 
  \label{equ: lin}
\end{equation}
where $\alpha_i^\pm$ are words in $x_i$ as described above.

Let  $\widetilde{X_K}^{(n)}$ be the $n^{th}$ cyclic cover of the knot complement $X_K$.  Its fundamental group $\pi_1(\widetilde{X_K}^{(n)}) \cong\text{Ker}(\pi_1(X_K)\rightarrow \mathbb{Z}_n)$ is an index $n$ subgroup of the knot group $\pi_1(X_K)$. Choose $\{Z^i\}_{i=0,\dots, n-1}$ to be the representative from each coset.  By applying the Reidemeister-Schreier Method \cite{ls} to the presentation (\ref{equ: lin}), we obtain a presentation of the  group $\pi_1(\widetilde{X_K}^{(n)})$  with

\begin{itemize}
 \item generators:  $Z^n$ and 
 
 $Z^k x_1 Z^{-k}$, ... , $Z^{k} x_{2g} Z^{-k}$ for  $k= 0, \cdots, n-1$;
 \item  relators: 
  \begin{equation}
   Z^{k+1} \alpha_i^+ Z^{-(k+1)}= Z^k\alpha_i^-Z^{-k}, \text{ for }  k=0, \cdots, n-2 \text{ and } i=1, \dots, 2g,
   \label{rel:11}
  \end{equation}
  \begin{equation}
   Z^n \cdot \alpha_i^+ \cdot Z^{-n} = Z^{n-1}\alpha_i^-Z^{-(n-1)},  \text{ for } i=1, \dots, 2g.
   \label{rel:12}
  \end{equation}
 \end{itemize}
In the presentation above,  $Z^k x_i Z^{-k}$ and $Z^n$ should be viewed as abstract symbols  rather than products of $Z$ and $x_i$.  Thus, words  $ Z^k \alpha_i^+ Z^{-k}$ as in (\ref{rel:11}) are  products of the generators  $Z^k x_i Z^{-k}$ and the word $Z^n \cdot \alpha_i^+ \cdot Z^{-n}$ in (\ref{rel:12}) is the product of $Z^{\pm n}$ and $x_i$. The notation is chosen to emphasize the fact that the isomorphism between the presented group and the subgroup $\text{Ker}(\pi_1(X_K)\rightarrow \mathbb{Z}_n)$ is given by  sending the abstract symbol $Z^k x_i Z^{-k}$ in the presentation to the element $Z^k x_i Z^{-k}$ of the knot group $\pi_1(X_K)$ for $k=0, \dots, n-1$ and $i=1,\dots,2g$.

Intuitively, this presentation can be understood in the following way. The $n^{th}$ cyclic cover  $\widetilde{X_K^{(n)}}$ can be constructed by gluing $n$ copies of $S^3-F\times (-1,1)$ together. We denote each copy by $Y_k$. Let $F_k$ be the Seifert surface associated with $Y_k$ and $F^{\pm}_k$ be $F_k\times \pm 1$ on $\partial Y_k$ for $k=0,\dots, n-1$. Then $Z^k x_i Z^{-k}$ are generator loops in $Y_k$ and each relation $Z^{k+1} \alpha_i^+ Z^{-(k+1)}= Z^k\alpha_i^-Z^{-k}$ in (\ref{rel:11}) is due to  the isomorphism between $\pi_1(F_k^-)$ and $\pi_1(F_{k+1}^+)$. In addition, the relation (\ref{rel:12}) is from the identification between $F_0^+$ and $F_{n-1}^-$.

\medskip

Now let's look at the fundamental group of  the $n^{th}$ cyclic branched cover $X_K^{(n)}$. From the construction of $X_K^{(n)}$, we have the following  isomorphism  $$\pi_1(X_K^{(n)})\cong\text{Ker}(\pi_1(X_K)\rightarrow \mathbb{Z}_n)/\ll Z^n\gg.$$ Therefore, the group $\pi_1(X_K^{(n)})$ inherits  the presentation with
\begin{itemize}
 \item generators: $Z^k x_1 Z^{-k}$, ... , $Z^{k} x_{2g} Z^{-k}$ for $k= 0, \cdots, n-1$;
 \item relators: 
  \begin{equation}
   Z^{k+1} \alpha_i^+ Z^{-(k+1)}= Z^{k}\alpha_i^-Z^{-k}, \text{ for }  k=0, \cdots, n-2 \text{ and } i=1, \dots, 2g,   \label{rel:21}
  \end{equation}
  \begin{equation}
   \alpha_i^+ = Z^{n-1}\alpha_i^-Z^{-(n-1)}, \text{ for } i=1, \dots, 2g.
   \label{rel:22}
  \end{equation}
 \end{itemize}
 
 \bigskip
 
 \begin{lemma}
   Given a  knot $K$ in $S^3$,  denote by $Z$ a meridional element in the knot group $\pi_1(X_K)$. Suppose that there exists a group homomorphism $\rho$ from $\pi_1(X_K)$ to a group $G$ and  $\rho(Z^{n})$ is in the center of $G$.  Then  $\rho$ induces a group homomorphism from $\pi_1(X_K^{(n)})$ to $G$. In particular, if $\rho$ is non-abelian, then the induced homomorphism is nontrivial. 
   \label{lem:in}
 \end{lemma}

\begin{proof}
Let $\rho|_{ker}$ be the restriction of  $\rho$ to the subgroup Ker$(\pi_1(X_K)\rightarrow \mathbb{Z}_n)$.  
We are going to show that the assignment 
 \begin{displaymath}
   Z^kx_iZ^{-k} \mapsto \rho|_{ker}(Z^kx_iZ^{-k}) \text{ for } i=1, \dots, 2g \text{ and } k=0,\dots, n-1  
 \end{displaymath}
also defines a homomorphism from $\pi_1(X_K^{(n)})$ to $G$.

First of all, the relations in (\ref{rel:11}) which are the same as the relations in (\ref{rel:21}) automatically hold. 
 It follows from (\ref{rel:12}) that  
\begin{displaymath}
   \rho|_{ker}(Z^n)\cdot \rho|_{ker}(\alpha_i^+)\cdot \rho|_{ker}(Z^{-n}) = \rho|_{ker}(Z^{n-1}\alpha_i^-Z^{-(n-1)}).
\end{displaymath}
Since by assumption $\rho|_{ker}(Z^n)=\rho(Z^n)$ is in the center of $G$, we have 
\begin{displaymath}
\rho|_{ker}(\alpha_i^+)= \rho|_{ker}(Z^n)\cdot \rho|_{ker}(\alpha_i^+)\cdot \rho|_{ker}(Z^{-n}) = \rho|_{ker}(Z^{n-1}\alpha_i^-Z^{-(n-1)}).
\end{displaymath}
That is, the relations in (\ref{rel:22}) hold as well. 

In addition,  if $\rho$ is a non-abelian homomorphism, since the commutator subgroup\\ $[\pi_1(X_K),\pi_1(X_K)]$ is the normal subgroup generated by $\{x_1, \dots, x_{2g}\}$, we have that $\rho(x_i)$ is not equal to the identity in $G$ for some $i$. Therefore, the induced homomorphism from $\pi_1(X_K^{(n)})$ to $G$ is nontrivial. 
 \end{proof}

\section{The left-orderability of the fundamental group $\pi_1(X_K^{(n)})$}

We finish the proof of Theorem \ref{thm:cri} in this section.  
\begin{theorem}
  Given any prime knot $K$ in $S^3$, denote by  $Z$ a meridional element of 
 $\pi_1(X_K)$. If there exists a non-abelian representation $\pi_1(X_K)$ to $SL(2,\mathbb{R})$ such that  $Z^n$ is sent to $\pm I$ then the fundamental group $\pi_1(X_K^{(n)})$ is left-orderable. 
 \label{thm:cri}
\end{theorem}

We will make use of the following criterion due to Boyer-Rolfsen-Wiest.   \begin{theorem}[\cite{brw}]
 Let $M$ be a compact, orientable, irreducible $3$-manifold. Then $\pi_1(M)$ is left-orderable, if there exists a nontrivial homomorphism from $\pi_1(M)$ to a left-orderable group.
 \label{thm:brw}
\end{theorem}

Note that the group $SL(2,\mathbb{R})$ itself is not left-orderable, but its universal covering group, denoted by $\widetilde{SL(2,\mathbb{R})}$, is left-orderable \cite{ber}. 
  Let $E$ be the covering map from $\widetilde{SL(2,\mathbb{R})}$ to $SL(2,\mathbb{R})$. Since $\widetilde{SL(2,\mathbb{R})}$ and $SL(2,\mathbb{R})$  are both connected, we have 
  \begin{displaymath}
   \mathcal{Z}(\widetilde{SL(2,\mathbb{R})})= E^{-1}(\mathcal{Z}(SL(2,\mathbb{R}))), 
  \end{displaymath}
where $\mathcal{Z}(\widetilde{SL(2,\mathbb{R})})$ and $\mathcal{Z}(SL(2,\mathbb{R}))$ are the centers of the Lie groups $\widetilde{SL(2,\mathbb{R})}$ and $SL(2,\mathbb{R})$ respectively \cite[p.~336]{hn}.  Therefore, $\mathcal{Z}(\widetilde{SL(2,\mathbb{R})})=E^{-1}(\{\pm I\})$. 

\begin{lemma}
Given any knot $K$ in $S^3$, let $Z$  be a meridional element in  the knot group
 $\pi_1(X_K)$. Suppose that  there exists a non-abelian $SL(2,\mathbb{R})$ representation of $\pi_1(X_K)$ such that  $Z^n$ is sent to $\pm I$. Then this representation induces a nontrivial $\widetilde{SL(2,\mathbb{R})}$ representation of  the fundamental group of the $n^{th}$ cyclic branched cover $\pi_1(X_K^{(n)})$ . 
 \label{lem:rep}
\end{lemma} 
\begin{proof}
The kernel of the covering map $Ker(E)$ is isomorphic to $\pi_1(SL(2,\mathbb{R}))\cong \mathbb{Z}$ and we have the following central extension
  \begin{displaymath}
   0 \longrightarrow \mathbb{Z} \longrightarrow \widetilde{SL(2,\mathbb{R})}\longrightarrow SL(2,\mathbb{R})\longrightarrow I.
 \end{displaymath}
  Suppose that  $\rho$ is a representation of $\pi_1(X_K)$ into $SL(2,\mathbb{R})$. Then the pullback 
  \begin{displaymath}
    \widetilde{SL(2,\mathbb{R})}\times_{SL(2,\mathbb{R})} \pi_1(X_K) =\{(M, x)\in  \widetilde{SL(2,\mathbb{R})}\times \pi_1(X_K): E(M)=\rho(x)\}, 
  \end{displaymath}
 is a central extension of $\pi_1(X)$ by $\mathbb{Z}$. 
 On the other hand,  
 $$H^2(\pi_1(X_K),\mathbb{Z})\cong H^2(X_K,\mathbb{Z})=0,$$ 
so every central extension of $\pi_1(X_k)$ by $\mathbb{Z}$ splits.  Hence,  $\rho$ can be  lifted to a representation into $\widetilde{SL(2,\mathbb{R})}$.  That is, the composition of a splitting map with the projection from $ \widetilde{SL(2,\mathbb{R})}\times_{SL(2,\mathbb{R})} \pi_1(X_K)$ to $\widetilde{SL(2,\mathbb{R})}$ is a lifting of $\rho$ \cite{wei} (also see \cite{ghy}).

Now assume that the representation $\rho$ of the knot group $\pi_1(X_K)$ satisfies the property  $\rho(Z^n)=\pm I$.  We  denote by $\tilde{\rho}$ a lifting of $\rho$.  Since $\rho(Z^n)=\pm I$,  we have $\tilde{\rho}(Z^n)$ is inside $E^{-1}(\pm I)$, which is equal to $\mathcal{Z}(\widetilde{SL(2,\mathbb{R})})$, the center of $\widetilde{SL(2,\mathbb{R})}$.
 \begin{center}
\setlength{\unitlength}{0.8cm}
\begin{picture}(14,3.5)
\thicklines
\put(6,1.7){$\tilde{\rho}$}
\put(8.2,2.4){$\widetilde{SL(2,\mathbb{R})}$}
\put(8.7,2.2){\vector(0,-1){1.4}}
\put(9,1.5){$E$}
\put(3,0.3){$\pi_1(X_K,P_-)$}
\put(5,0.8){\vector(2,1){3}}
\put(8.2,0.2){$SL(2,\mathbb{R})$}
\put(5.5,0.5){\vector(1,0){2.5}}
\put(7,0.7){$\rho$}
\end{picture}
\end{center}
In addition,  if $\rho$ is a non-abelian representation, then $\tilde{\rho}$ is non-abelian. By Lemma \ref{lem:in}, the representation  $\tilde{\rho}$ induces a nontrivial  $\widetilde{SL(2,\mathbb{R})}$ representation of $\pi_1(X_K^{(n)})$. 
\end{proof}

\begin{proof}[Proof of Theorem \ref{thm:cri} ]
Let $\rho$ be a non-abelian $SL(2,\mathbb{R})$ representation of the knot group $\pi_1(X_K)$, with $\rho(Z^n)=\pm I$. By Lemma \ref{lem:rep}, this representation induces a nontrivial   $\widetilde{SL(2,\mathbb{R})}$ representation of the group $\pi_1(X_K^{(n)})$.

 The  group $\widetilde{SL(2,\mathbb{R})}$ can be  embedded inside the group of order-preserving homeomorphisms of $\mathbb{R}$, so it is left-orderable  \cite{ber}. Moreover, the $n^{th}$ cyclic branched cover $X_K^{(n)}$ is irreducible if $K$ is a prime knot \cite{plo}. 
 Thus,  Theorem \ref{thm:cri} follows from Theorem \ref{thm:brw}.
\end{proof}

 \section{An Application to $(p,q)$ two-bridge knots, with $p\equiv 3$ mod $4$}
 In this section  we  apply  Theorem \ref{thm:cri} to $(p,q)$ two-bridge knots with $p=3 \text{ mod } 4$. 
 We show that given any two-bridge knot of this type, the fundamental group of the $n^{th}$ cyclic branched cover  is  left-orderable  if  $n$  is sufficiently large. 
 
 Let $K$ be a $(p,q)$ two-bridge knot.  From the Schubert normal form \cite[p.~21]{kaw},  the knot group has a presentation of the following form:
 \begin{displaymath}
  \pi_1(X_K)=<x,y: wx=yw>, 
   \end{displaymath}
 \begin{displaymath}
  \text{ where }
 w=(x^{\epsilon_1}y^{\epsilon_2})\dots (x^{\epsilon_{p-2}}y^{\epsilon_{p-1}}) \text{ and } \epsilon_i=\pm 1.
 \end{displaymath}

 Set $\rho:\pi_1(X_K)\rightarrow SL(2, \mathbb{C})$ be a non-abelian  representation of the knot group into $SL(2,\mathbb{C})$. Up to conjugation, we can assume that
\begin{equation}
\rho(x)= \begin{pmatrix}
 m & 1\\
 0& m^{-1}\\
 \end{pmatrix},\quad \rho(y)=\begin{pmatrix}
  m&0\\
  s& m^{-1}\\
 \end{pmatrix}.
 \label{equ:rep}
\end{equation}

Hence, $\rho(w)=\rho(x)^{\epsilon_1}\rho(y)^{\epsilon_2} \dots \rho(x)^{\epsilon_{p-2}}\rho(y)^{\epsilon_{p-1}}$ is a matrix with entries in $\mathbb{Z}[m^{\pm 1}, s]$. Denote $\rho(w)=\begin{pmatrix}
 w_{11} & w_{12}\\
 w_{21} & w_{22}
\end{pmatrix}$,  $w_{ij}\in \mathbb{Z}[m^{\pm 1}, s]$.  

From the group relation $wx=yw $, we have
 \begin{displaymath}
 \begin{pmatrix}
 w_{11} & w_{12}\\
 w_{21} & w_{22}
\end{pmatrix} \begin{pmatrix}
 m & 1\\
 0& m^{-1}\\
 \end{pmatrix}=\begin{pmatrix}
  m&0\\
  s& m^{-1}\\
 \end{pmatrix}\begin{pmatrix}
 w_{11} & w_{12}\\
 w_{21} & w_{22}
\end{pmatrix}. 
\end{displaymath}

This is equivalent to 
\begin{equation}
\begin{pmatrix}
   0 & w_{11}+(m^{-1}-m)w_{12}\\
   (m-m^{-1})w_{21}-sw_{11} & w_{21}-sw_{12}
 \end{pmatrix}=0
\end{equation}

and  hence $s$ and $m$ must satisfy the equation 
 $$w_{11}+(m^{-1}-m)w_{12}=0.$$ 
 
 In \cite{ril2}, it is shown that the above equation  is also a sufficient condition. 
\begin{proposition}[Theorem 1 of \cite{ril2}]
  The assignment of $x$ and $y$ as in (\ref{equ:rep}) defines a non-abelian $SL(2,\mathbb{C})$ representation of the knot group $$\pi_1(X_K)=<x,y: wx=yw>$$  if and only if 
  \begin{equation}
    \varphi(m,s)\triangleq w_{11}+(m^{-1}-m)w_{12}=0.
      \label{equ:ril}
  \end{equation}
  \label{prop:ril}
\end{proposition}
\bigskip
 
We need to make use of several properties of the polynomial $\varphi(m,s)$. All of these properties are either proven or claimed throughout Riley's paper \cite{ril2}.  For readers' convenience, we organize them and provide a proof in the following lemma.
 
\begin{lemma}[cf. \cite{ril2}]
The polynomial $\varphi(m,s)$  in $\mathbb{Z}[m^{\pm 1}, s]$ satisfies the following: 
\begin{enumerate}
 \item As a polynomial in $s$ with coefficients in $\mathbb{Z}[m^{\pm 1}]$, $\varphi(m,s)$ has $s$-degree  equal to $\frac{p-1}{2}$, with the leading coefficient $\pm 1$.
  \item $\varphi(1,0)\neq 0$.
  \item $\varphi(m,s)$ does not have repeated factors. 
 \item $\varphi(m,s)=\varphi(m^{-1},s)$ and thus $\varphi(m,s)=f(m+m^{-1},s)$ where $f$ is a two-variable polynomial with coefficients in $\mathbb{Z}$. 
\end{enumerate}
\label{lem:ril}
\end{lemma}

\begin{proof}
\begin{enumerate}
 \item Since we assign 
\begin{displaymath}
\rho(x)=\begin{pmatrix}
 m & 1\\
 0&m^{-1}
\end{pmatrix},\quad \rho(y)=\begin{pmatrix}
 m & 0\\
 s & m^{-1}
\end{pmatrix},
\end{displaymath}
 through a direct computation we have 
 \begin{displaymath}
  \rho(xy)=\begin{pmatrix}
   m^2+s&m^{-1}\\
   m^{-1}s & m^{-2}
  \end{pmatrix}, \quad \rho(x^{-1}y)=\begin{pmatrix}
   1-s & -m^{-1}\\
   ms & 1
  \end{pmatrix},
 \end{displaymath}
 \begin{displaymath}
 \rho( xy^{-1})=\begin{pmatrix}
   1-s & m\\
   -m^{-1}s &1
  \end{pmatrix}, \quad \rho(x^{-1}y^{-1})=\begin{pmatrix}
   m^{-2}+s& -m\\
   -ms&m^2
  \end{pmatrix}.
 \end{displaymath}

Say a matrix  $A$  in $M_2(\mathbb{Z}[m^{\pm 1},s])$ has $s$-degree equal to  $n$ if
\begin{displaymath}
A= \begin{pmatrix}
  \pm s^n+f_{11}(m, s) & f_{12}(m,s)\\
  f_{21}(m,s) & f_{22}(m,s)
 \end{pmatrix}, \text{ where }
\end{displaymath}
the $s$-degrees of $f_{11}$, $f_{12}$ and $f_{22}$ are strictly less than $n$ and the $s$-degree of $f_{21}$ is less than or equal to $n$.  Hence the matrices $\rho(xy)$, $\rho(x^{-1}y)$, $\rho(xy^{-1})$ and $\rho(x^{-1}y^{-1})$ all have $s$-degrees equal to $1$. Moreover,  the product of an $s$-degree $n$ matrix and an $s$-degree $m$ matrix is an $s$-degree $m+n$ matrix.
 Since
 \begin{displaymath}
   w=(x^{\epsilon_1}y^{\epsilon_2})\dots (x^{\epsilon_{p-2}}y^{\epsilon_{p-1}}), \text{ with } \epsilon_i=\pm 1, 
 \end{displaymath}
 we have that the matrix $$\rho(w)=\begin{pmatrix}
 w_{11} & w_{12}\\
 w_{21} & w_{22}
\end{pmatrix}$$ is a product of $\frac{p-1}{2}$ $s$-degree $1$ matrices. Therefore, the matrix $\rho(w)$ has  $s$-degree equal to $\frac{p-1}{2}$.  That is, the entry $w_{11}$ has   $\pm s^{n}$ as the leading term and the $s$-degree of $w_{12}$ is strictly less than $\frac{p-1}{2}$. As a result, $\varphi(m,s)= w_{11}+(m^{-1}-m)w_{12}$ has leading term equal to $\pm s^{n}$. 

\bigskip 

\item Notice that as $m=1$ and $s=0$, we have 
\begin{displaymath}
  \rho(x)=\begin{pmatrix}
   1&1\\
   0 & 1
  \end{pmatrix}, \quad \rho(y)=\begin{pmatrix}
   1 & 0\\
   0 & 1
  \end{pmatrix}.
 \end{displaymath}
This assignment can not define a representation of the knot group
  $$\pi_1(X_K)=<x,y: wx=yw>,$$
  because these two matrices $ \rho(x)=\begin{pmatrix}
   1&1\\
   0 & 1
  \end{pmatrix}$ and  $\rho(y)=\begin{pmatrix}
   1 & 0\\
   0 & 1
  \end{pmatrix}$ are not conjugate to each other. Therefore, $\varphi(1,0)\neq 0$ by the Proposition \ref{prop:ril}.
  
  \bigskip
  
  \item  Let $\Delta_K(t)$ be the Alexander polynomial of the knot $K$. It is shown in \cite[Proposition 1.1, Theorem 1.2]{na} (also see \cite{lin, bf}) that any knot group has $\frac{|\Delta_K(-1)|-1}{2}$  irreducible $SL(2,\mathbb{C})$ metabelian representations up to conjugation and that these metabelian representations send meridional elements to matrices of eigenvalues $\pm i$. For a $(p,q)$ two-bridge knot, $p$ equals  $|\Delta_K(-1)|$. This implies that the degree $\frac{p-1}{2}$ polynomial equation $\varphi(i, s)=0$ has $\frac{p-1}{2}$ distinguished roots. Therefore $\varphi(i,s)$ does not have repeated factors and so is $\varphi(m,s)$.  
  
Note that  we can also use the fact that $\varphi(1,s)$ does not have any repeated factors to prove that $\varphi(m,s)$ has no repeated factors \cite[Theorem 3]{ril1}.

\bigskip

\item
   Assume that the assignment
  \begin{displaymath}
   \rho(x)=\begin{pmatrix}
    m & 1\\
    0&m^{-1}
   \end{pmatrix}, \quad \rho(y)=\begin{pmatrix}
    m & 0\\
    s & m^{-1}
   \end{pmatrix}
  \end{displaymath}
 defines a representation of the knot group $$\pi_1(X_K)=<x,y : wx =yw>.$$ Then 
 \begin{displaymath}
   \rho'(x)= P\begin{pmatrix}
    m & 1\\
    0&m^{-1}
   \end{pmatrix}P^{-1} = \begin{pmatrix}
    m^{-1} & 1\\
    0&m
   \end{pmatrix}
   \end{displaymath}
   \begin{displaymath}
   \rho'(y)=P\begin{pmatrix}
    m & 0\\
    s & m^{-1}
   \end{pmatrix}P^{-1} = \begin{pmatrix}
    m^{-1} & 0\\
    s & m
   \end{pmatrix}
 \end{displaymath}
 also defines a representation,
 where  \begin{displaymath}
 P= \begin{pmatrix}
    1 & (m^{-1}-m)/s\\
    m-m^{-1} & 1
   \end{pmatrix}.
 \end{displaymath}
The matrix $P$ is  well-defined and invertible whenever $(m,s)$ is not in  the finite set 
$$S\triangleq\{(m,s): s= 0 , \varphi(m,s)=0\}\cup\quad\quad\quad\quad\quad$$
$$\quad\quad\quad\quad\quad\quad\{(m,s): s=-(m-m^{-1})^2, \varphi(m,s)=0\}.$$ 

The set $S$ is finite because neither $\varphi(m,0)$ nor $\varphi(m, -(m-m^{-1})^2)$ is a  zero polynomial. Otherwise,  $(1,0)$ will be a solution for $\varphi(m,s)$, which contradicts part $(2)$.

\bigskip 

Denote by $V(g)$ the solution set of  a polynomial $g$. As we described above, 
\begin{displaymath}
 V(\varphi(m,s)) - S \subset V(\psi(m,s)),
\end{displaymath}
where $\psi(m,s)=\varphi(m^{-1},s)$. Points in $S$ are not  isolated, since they are embedded inside the algebraic curve $V(\varphi(m,s))$. By continuity, we have 
\begin{displaymath}
 V(\varphi(m,s)) \subset V(\psi(m,s)).
\end{displaymath}

By  part $(3)$, neither of $\varphi(m,s)$ and $\psi(m,s)$ has repeated factors, so 
 the  ideal $<\psi(m,s)>$  is contained inside the ideal  $<\varphi(m,s)>$ in $\mathbb{Z}[m^{\pm 1}, s]$.
 On the other hand, both $\varphi(m,s)$ and $\psi(m,s)$ have the same leading term, which is either $s^{(p-1)/2}$ or $-s^{(p-1)/2}$, so $\varphi(m,s)= \psi(m,s)=\varphi(m^{-1}, s)$.

\end{enumerate}
\end{proof}

Now we are ready to prove the main result.  

\begin{theorem}
\label{thm:tb}
Given a $(p,q)$ two-bridge knot $K$, with $p\equiv 3 \text{ mod } 4$, there are only  finitely many cyclic branched covers, whose fundamental groups  are not left-orderable.
\end{theorem}
\begin{proof}
 We are going to show that for sufficiently large $n$, the group  $\pi_1(X_K)$ has a non-abelian $SL(2,\mathbb{R})$ representation with $x^n$ sent to $-I$. 
 
 As before, we assign
 \begin{displaymath}
  \rho(x)=\begin{pmatrix}
   m & 1\\
   0 &m^{-1}
  \end{pmatrix}, \quad \rho(y)=\begin{pmatrix}
   m & 0\\
   s & m^{-1}
  \end{pmatrix}.
 \end{displaymath}

Let $m=e^{i \theta}$. Since $p=3 \text{ mod } 4$, 
by Lemma $4.2$, we have that 
 $\varphi(e^{i\theta},s)$
is an odd degree real polynomial in $s$.  So for any given $\theta$, the equation $\varphi(e^{i\theta}, s)=0$ has at least one real solution for $s$. 
We assume that  $s_0$ is a real solution of the equation $\varphi(1, s)=0$. It is known that  the polynomial $\varphi(1,s)$ does not have repeated factors \cite[Theorem 3]{ril1}. Hence, $\varphi_s(e^{i\theta}, s)|_{\theta=0,s=s_0} \neq 0$ and locally there exists a real function $s(\theta)$ such that $\varphi(e^{i \theta},s(\theta))=0$ and $s(0)=s_0$.
 
 \bigskip
 Consider the following one-parameter family of non-abelian representations.

\begin{displaymath}
  \rho\{\theta\}(x)=\begin{pmatrix}
   e^{i\theta} & 1\\
   0 & e^{-i\theta}
  \end{pmatrix}, \quad \rho\{\theta\}(y)=\begin{pmatrix}
   e^{i\theta} & 0\\
   s(\theta) & e^{-i\theta}
  \end{pmatrix}.
 \end{displaymath}
 
 As $\theta \neq 0$, the representations $\rho\{\theta\}$ can be diagonalized to the following forms which we still denote by $\rho\{\theta\}$,  
 
 \begin{equation}
\rho\{\theta \}(x)= \begin{pmatrix}
 e^{i \theta} & 0\\
 0& e^{-i\theta}\\
 \end{pmatrix},\quad \rho\{\theta\}(y)=\begin{pmatrix}
  e^{i\theta}-\frac{s(\theta)}{2 \sin(\theta)} i&-1+\frac{s(\theta)}{4 \sin^2(\theta)}\\
  s(\theta)&e^{-i\theta}+\frac{s(\theta)}{2 \sin(\theta)} i
 \end{pmatrix}.
 \label{equ:rep2}
\end{equation}

According to \cite[p.~786]{kho}, this representation can be conjugated to an $SL(2,\mathbb{R})$ representation if and only if 
 \begin{equation}
\text{ either }  s(\theta)<0 \text{ or } s(\theta)>4 \sin^2(\theta).
  \end{equation}
We can verify this via a direction computation. In fact, when $s<0$ or $s>4 \sin^2(\theta)$, the representation $\rho\{\theta\}$ is conjugate to an $SU(1,1)$  representation by the matrix

\begin{displaymath}
 \begin{pmatrix}
  \sqrt{\frac{1}{\sqrt{t}}+t} & t\\
  \sqrt{t} & \sqrt{\sqrt{t}+t^2}
 \end{pmatrix}, \text{ where } t=\frac{1}{4\sin^2(\theta)}-\frac{1}{s} \text{ is positive},
\end{displaymath}
and $SU(1,1)$ is conjugate to $SL(2,\mathbb{R})$  via the matrix $\begin{pmatrix}
  1& -i\\
  1 & i
\end{pmatrix}$ in $GL(2,\mathbb{C})$.

On the other hand,  
\begin{displaymath}
 \lim_{\theta\rightarrow 0} s(\theta) =s_0, \text{ where $s_0$ is not equal to }0 \text{ by Lemma \ref{lem:ril} part }(2).
\end{displaymath}

Hence, when $\theta$ is small enough, either $s(\theta)<0$ or $s(\theta)>4 \sin^2(\theta)$. 
Now let  $\theta=\pi/n$.  For sufficiently large $n$, the non-abelian representation $\rho\{\theta\}$ as in (\ref{equ:rep2})  satisfies $\rho\{\theta\}(x)^n=-I$ and  conjugates to an $SL(2,\mathbb{R})$ representation. Therefore,  by  Theorem \ref{thm:cri}, the conclusion follows.
\end{proof}

We are ending this paper by computing one specific  example. 
\begin{example}
  We consider the two bridge knot $(7, 4)$, which is listed as $5_2$ in Rolfsen's table.  The fundamental group $\pi_1(X_{5_2})$ has a presentation 
 \begin{displaymath}
  \pi_1(X_{5_2}) =<x,y:wx=yw>,
\end{displaymath}
where $w=xyx^{-1}y^{-1}xy$.

From this presentation, we can compute the polynomial 
$$\varphi(m,s)=s^3+(2(m^2+m^{-2})-3)s^2+((m^4+m^{-4})-3(m^2+m^{-2})+6)s+2(m^2+m^{-2})-3.$$ 
as defined in (\ref{equ:ril}).  And 
\begin{displaymath}
 \varphi(e^{i\theta},s)= s^3+(4\text{ cos}(2\theta)-3)s^2+(2\text{ cos}(4\theta)-6 \text{ cos}(2\theta)+6)s+4\text{ cos}(2\theta)-3,
\end{displaymath}
which is a real polynomial in $s$ with degree $3$. Hence, we can solve a closed formula for $s(\theta)$ such that $\varphi(e^{i\theta}, s(\theta))=0$. Figure \ref{fig:q} is the graph of the solution $s(\theta)$ on the interval $\theta\in [0,1]$.

\bigskip

\begin{figure}[H]
\begin{center} 
\includegraphics[height=4cm,width=9cm]{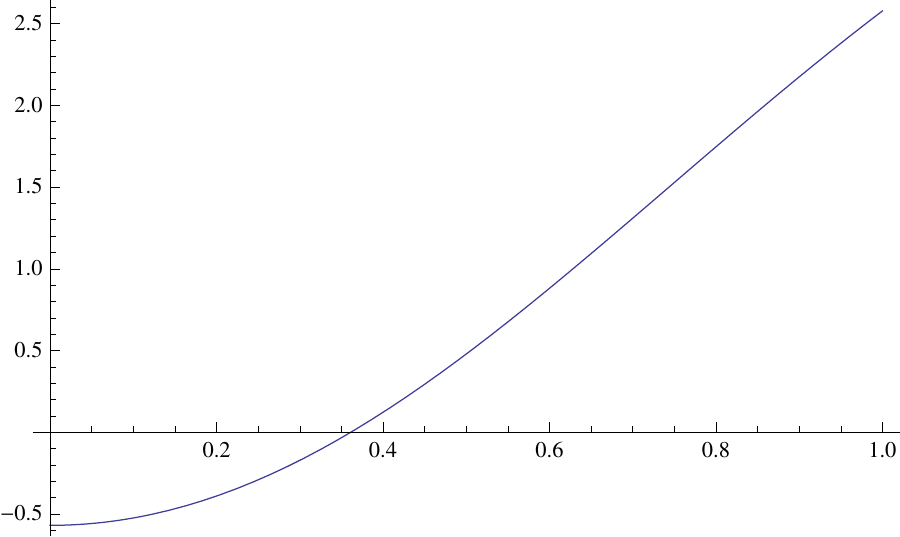} 
\caption{\small \sl }
\label{fig:q}
\end{center} 
\end{figure}
\end{example}

In particular, when $n=9$, we have that $\frac{\pi}{9}\approx 0.349$ and $s(\frac{\pi}{9})\approx -0.03667$.  The group  $\pi_1(X_{5_2}^{(n)})$ is left-orderable  when $n\geq 9$. For cyclic branched covers $X_{5_2}^{(n)}$ with $n<9$, the other known cases are $n=2,3$ \cite{dpt} and $n=4$ \cite{gl}, none of which has a left-orderable fundamental group.

{\small \subsection*{Acknowledgment } The author would like to thank Oliver Dasbach for drawing her attention to the topic of the current paper and his consistent encouragement and support throughout her graduate study.  Also, the author  would  like to give thanks to Michel Boileau, Tye Lidman and Neal Stoltzfus for helpful conversation and suggestions. Finally, she gives many thanks to the referee for pointing out the similarity between Theorem \ref{thm:cri} and \cite[Theorem 6]{bgw} and his or her many helpful comments.}

\bibliographystyle{alpha}
\bibliography{reference}

\end{document}